\newcommand{\algten}{\mathop{\underline{\otimes}}}
\newcommand{\bop}[2]%
{\ifthenelse{\equal{#2}{}}{\bopp( #1 )}{\bopp( #1; #2 )}}
\newcommand{\bopp}{B}
\newcommand{\bra}[1]{\langle #1 |}
\newcommand{\comp}{\mathop{\circ}}
\newcommand{\dom}{\mathop{\mathrm{Dom}}}
\newcommand{\elltwo}{L^2( \R_+; \mul )}
\newcommand{\eevec}{\varepsilon}
\newcommand{\evec}[1]{\mathord{\eevec(#1)}}
\newcommand{\expn}{\mathbb{E}}
\newcommand{\filt}{\mathcal{B}}
\newcommand{\fock}{\mathcal{F}}
\newcommand{\hilb}{\mathsf{H}}
\newcommand{\hlf}{\mbox{$\frac12$}}
\newcommand{\id}{I}
\newcommand{\idop}{\mathop{\mathrm{id}}\nolimits}
\newcommand{\indf}[1]{1_{#1}}
\newcommand{\ini}{\mathfrak{h}}
\newcommand{\intd}{\,\rd}
\newcommand{\ket}[1]{| #1 \rangle}
\newcommand{\mmul}{{\widehat{\mul}}}
\newcommand{\mul}{\mathsf{k}}
\newcommand{\semigroup}{\mathcal{P}}
\newcommand{\semigrouq}{\mathcal{Q}}
\newcommand{\sstex}{\mathfrak{z}}
\newcommand{\stex}[1]{\mathord{\sstex(#1)}}
\newcommand{\uwkten}{\mathop{\overline{\otimes}}}
\newcommand{\Vac}{\Omega}
\newcommand{\vac}{\omega}
\newcommand{\vNa}{\mathsf{A}}
\newcommand{\vNnoise}{\mathsf{N}}
\newcommand{\Wiener}{{\mathbb{W}}}
\newcommand{\rd}{\mathrm{d}}

\newcommand{\Cf}{\textit{Cf}.}
\newcommand{\cf}{\textit{cf}.}
\newcommand{\ie}{\textit{i.e., }}
\renewcommand{\ge}{\geqslant}
\renewcommand{\le}{\leqslant}
\newcommand{\via}{\textit{via }}

\newcommand{\C}{\mathbb{C}}
\newcommand{\I}{\mathrm{i}}
\newcommand{\R}{\mathbb{R}}

\newenvironment{mylist}%
{\begin{list}{}%
{\leftmargin 7em\labelwidth 5em\rightmargin 1.25em%
\topsep 0.75ex\itemsep 0.5ex}}%
{\end{list}}

\documentclass[10pt,draft,reqno]{amsart}
     \makeatletter
     \def\section{\@startsection{section}{1}%
     \z@{.7\linespacing\@plus\linespacing}{.5\linespacing}%
     {\bfseries
     \centering
     }}
     \def\@secnumfont{\bfseries}
     \makeatother
\newtheorem{theorem}{Theorem}[section]
\newtheorem{lemma}[theorem]{Lemma}
\newtheorem{proposition}[theorem]{Proposition}
\newtheorem{corollary}[theorem]{Corollary}
\theoremstyle{definition}
\newtheorem{definition}[theorem]{Definition}

\theoremstyle{remark}
\newtheorem{remark}[theorem]{Remark}
\numberwithin{equation}{section}
\usepackage{amsmath,amssymb,ifthen}

\begin{document}

\title[Quantum Feynman--Kac formulae]%
{A vacuum-adapted approach to\\quantum Feynman--Kac formulae}

\author{Alexander C. R. Belton}
\address{ACRB: Department of Mathematics and
Statistics, Lancaster University, Lancaster LA1 4YF, United Kingdom}
\email{a.belton@lancaster.ac.uk}

\author[J. Martin Lindsay]{J. Martin Lindsay}
\address{JML: Department of Mathematics and Statistics,
Lancaster University, Lancaster LA1 4YF, United Kingdom}
\email{j.m.lindsay@lancaster.ac.uk}

\author[Adam G. Skalski]{Adam G. Skalski}
\address{AGS: Mathematical Institute of the Polish Academy
of Sciences, ul.~\'{S}niadeckich~8, P.O. Box 21, 00-956 Warszawa,
Poland}
\email{a.skalski@impan.pl}



\begin{abstract}
The vacuum-adapted formulation of quantum stochastic calculus is
employed to perturb expectation semigroups \via a Feynman--Kac
formula. This gives an alternative perspective on the perturbation
theory for quantum stochastic flows that has recently been developed
by the authors.
 \end{abstract}

\maketitle

\section{Introduction}

Let $\alpha = ( \alpha_t )_{t \in \R}$ be an ultraweakly continuous
group of normal $*$-automorphisms of a von~Neumann algebra $\vNa$
acting faithfully on the Hilbert space $\ini$, and let $\delta$ be its
ultraweak generator. Gaussian subordination may be used to construct
an ultraweakly continuous semigroup $\semigroup^0$ on $\vNa$ with
ultraweak pre-generator $\hlf \delta^2$ \cite[Section~1]{LiS97} in the
following manner. If $B = ( B_t )_{t \ge 0}$ is standard Brownian
motion on Wiener's probability space $\Wiener$ then, by It\^o's
formula, the unital $*$-homomorphism
\[
j_t : \vNa \to \vNa \uwkten L^\infty( \Wiener ) = %
L^\infty( \Wiener ; \vNa); \ %
a \mapsto \alpha_{B_t( \, \cdot \, )}( a ) \qquad ( t \ge 0 )
\]
satisfies the stochastic differential equation
\begin{equation}\label{eqn:ls}
j_t( x ) = x + \int_0^t j_s( \delta( x ) ) \intd B_s + %
\frac12 \int_0^t j_s( \delta^2( x ) ) \intd s %
\qquad ( x \in \dom \delta^2 )
\end{equation}
in the strong sense on $L^2( \Wiener; \ini )$. Thus
\[
\semigroup^0_t( a ) u := \expn_\Wiener[ j_t( a ) u ] \qquad %
( a \in \vNa, \, u \in \ini \subset L^2( \Wiener; \ini ) )
\]
defines an ultraweakly continuous semigroup
$( \semigroup^0_t )_{t \ge 0}$ of normal unital completely positive
contractions on~$\vNa$ whose ultraweak generator is as desired.

For the case where $\alpha$ is unitarily implemented, Lindsay and
Sinha obtained an ultraweakly continuous semigroup $\semigroup^b$
with Feynman--Kac representation
\begin{equation}
\label{FK LS}
\semigroup^b_t( a ) u = \expn_\Wiener[ j_t( a ) m^b_t u ] %
\qquad ( t \ge 0, \, a \in \vNa, \, u \in \ini )
\end{equation}
whose ultraweak generator extends $\hlf \delta^2 + \rho_b \delta$,
where $\rho_b : a \mapsto a b$ is the operator on~$\vNa$ of right
multiplication by $b$ \cite[Theorem~3.2]{LiS97}. Here $m^b$ is the
exponential martingale such that
\[
 m^b_t = \id + \int_0^t j_s( b ) m^b_s \intd B_s \qquad ( t \ge 0 ),
\]
where $b \in \vNa$ is self adjoint. For the Laplacian on~$\R^{3 d}$
and the commutative von~Neumann algebra $L^\infty( \R^{3 d} )$, such
vector-field perturbations were studied from this viewpoint by
Parthasarathy and Sinha (\cite{PaS83}). Other works on quantum
Feynman--Kac formulae include \cite{Acc78}, \cite{HIP8284},
\cite{AcF83} and \cite{Arv84}, all of which belong to the
pre-quantum stochastic era. The classical Feynman-Kac formula for
Schr\"odinger operators, which is closely related to instances of
the Trotter product formula, is well described in the books
\cite{RSv2} and \cite{Simon}.

The results of Lindsay and Sinha have been fully generalised in
\cite{BLS11}. In that paper a general perturbation theory for quantum
stochastic flows is developed, yielding a much wider class of quantum
Feynman--Kac formulae. Here we take our inspiration from
\cite{BaP03}. The semigroups defined in \eqref{FK LS} will not, in
general, be positive or even real (\ie $*$-preserving). In this light
Bahn and Park investigate a more symmetric form of Feynman--Kac
perturbation, using instead an operator process~$n^b$ such that
\begin{equation}\label{eqn:bp}
n^b_t f = f + %
\int_0^t j_s( b ) \expn_\Wiener[ n^b_s f | \filt_s ] \intd B_s - %
\frac12 \int_0^t j_s( b^2 ) \expn_\Wiener[ n^b_s f | \filt_s ] \intd s
\end{equation}
for all $f \in L^2( \Wiener; \ini )$, where $( \filt_t )_{t \ge 0}$ is
the canonical filtration of the Brownian motion $B$. In this case,
letting
\[
\semigrouq^b_t( a ) u := %
\expn_\Wiener[ ( n^b_t )^* j_t( a ) n^b_t u ] %
\qquad ( a \in \vNa, \, u \in \ini )
\]
gives an ultraweakly continuous completely positive semigroup
$( \semigrouq^b_t )_{t \ge 0}$ on~$\vNa$, which is contractive
if~$n^b$ is and whose generator extends the map
\begin{equation}\label{eqn:BP}
\hlf \delta^2 + \lambda_b \delta + \rho_b \delta + %
\lambda_b \rho_b - \hlf \lambda_{b^2} - \hlf \rho_{b^2},
\end{equation}
where $\lambda_b$ denotes the operator on $\vNa$ given by left
multiplication by $b$.

In this work we are guided by the form of~\eqref{eqn:bp}; the
conditional expectations make it reminiscent of a stochastic
differential equation used by Alicki and Fannes for dilating quantum
dynamical semigroups \cite[Equation (12)]{AlF87}. As observed in
\cite{Blt01}, this type of equation may be profitably interpreted in
the vacuum-adapted form of quantum stochastic calculus. In contrast
to \cite{BLS11}, where the standard identity-adapted
(Hudson--Parthasarathy) theory is used, here the analysis is slightly
easier although the algebra becomes a bit more complicated.

We describe the contents of the paper next, restricting our
description here to the one-dimensional case, for simplicity. The
requirement that $\alpha$ is unitarily implemented is removed; our
primary object is a vacuum-adapted quantum stochastic flow. This is
an ultraweakly continuous family $j = ( j_t )_{t \ge 0}$ of normal
$*$-homomorphisms which form a vacuum-adapted quantum stochastic
cocycle on Boson Fock space over $L^2( \R_+ )$ and which are as unital
as vacuum adaptedness permits. The flow $j$ is assumed to satisfy the
quantum stochastic differential equation
\begin{equation}\label{eqn:iqsde}
\rd j_t( x ) = %
j_t( \delta_0( x ) ) \intd A^\dagger_t + %
j_t( \pi_0( x ) ) \intd \Lambda_t + %
j_t( \delta^\dagger_0( x ) ) \intd A_t + %
j_t( \tau_0( x ) ) \intd t
\end{equation}
for all $x \in \vNa_0$, where $\vNa_0$ is a subset of $\vNa$, and
the \emph{structure maps}
\[
\tau_0, \ \delta_0, \ \delta_0^\dagger, \ \pi_0 : \vNa_0 \to \vNa
\]
must satisfy certain algebraic relations, thanks to the unital and
$*$-homomorphic properties of $j$. Equation~\eqref{eqn:iqsde}
generalises~\eqref{eqn:ls}, which corresponds to the case
where~$\vNa_0 = \dom \delta^2$, $\pi_0$ is the inclusion map,
\[
\delta_0 = \delta^\dagger_0 = \delta |_{\vNa_0} %
\qquad \mbox{and} \qquad \tau_0 = \hlf \delta^2.
\]
The appearance of the non-zero gauge term $\pi_0$ is due to the fact
that we are working in the vacuum-adapted set-up:
\cf~\cite[Theorem~7.3]{Blt10}. It follows from~\eqref{eqn:iqsde} that
the quantum stochastic flow satisfies the equation
\[
\langle u \Vac, j_t( x ) v \Vac \rangle = %
\langle u, v \rangle + %
\int_0^t \langle u \Vac, j_s( \tau_0( x ) ) v \Vac \rangle \intd s %
\qquad ( u, v \in \ini, \, t \ge 0, \, x \in \vNa_0 ),
\]
where $\Vac$ denotes the Fock vacuum vector. The generator of the
\emph{vacuum-expectation semigroup}
$\semigroup^0 := ( \expn \comp j_t )_{t \ge 0}$ therefore extends the
map~$\tau_0$. A natural assumption here is that~$\tau_0$ is a
pre-generator of $\semigroup^0$, however our results do not require
it.

Starting with Evans and Hudson \cite{EvH90}, several authors have used
conjugation with a unitary process to perturb quantum stochastic
flows. These works focused on the case of bounded structure maps, so
that the vacuum-expectation semigroup~$\semigroup^0$ is norm
continuous, and considered identity-adapted flows and processes. For
$h = h^* \in \vNa$ and $l \in \vNa$ there exists a unitary process $U$
such that
\[
U_0 = \id, \qquad \rd U_t = %
j_t( l ) U_t \intd A^\dagger_t + %
j_t( -l^* ) U_t \intd A_t +
j_t( -\I h - \hlf l^* l ) U_t \intd t,
\]
and the vacuum-expectation semigroup of the perturbed flow
$( a \mapsto U_t^* j_t( a ) U_t )_{t \ge 0}$ has generator
\[
\tau_0 + %
\lambda_{l^*} \delta_0 + \rho_l \delta^\dagger_0 + %
\lambda_{l^*} \rho_l \pi_0 + %
\I [ h, \, \cdot \, ] - \hlf \{ l^* l, \, \cdot \, \},
\]
where $[ \, \cdot \, , \, \cdot \, ]$ and
$\{ \, \cdot \, , \, \cdot \, \}$ denote commutator and
anticommutator. The main result obtained here includes this situation
as a special case.

For any vacuum-adapted quantum stochastic flow $j$ and any
$c = %
\left[\begin{smallmatrix} c_0 \\[0.5ex] c_1 \end{smallmatrix}\right]$
in $\vNa \oplus \vNa$, Theorem~\ref{thm:main} below gives a
process~$M^c$ such that $M^c - \id$ is vacuum adapted and the
following quantum stochastic differential equation is satisfied:
\[
\rd ( M^c - \id )_t = %
j_t( c_0 ) M^c_t \intd t + j_t( c_1 ) M^c_t \intd A^\dagger_t.
\]
Consequently, for any
$d = \smash[b]{%
\left[\begin{smallmatrix} d_0 \\[0.5ex] d_1 \end{smallmatrix}\right]}$
in $\vNa \oplus \vNa$, there is an ultraweakly continuous semigroup
$\semigroup^{c, d}$ on~$\vNa$ with
\[
\langle u, \semigroup^{c, d}_t( a ) v \rangle = %
\langle u \Vac, (M^c_t)^* j_t( a ) M^d_t v \Vac \rangle %
\qquad ( u, v \in \ini, \, t \ge 0, \, a \in \vNa ).
\]
When $j$ satisfies~\eqref{eqn:iqsde}, the ultraweak generator
of~$\semigroup^{c, d}$ necessarily extends
\begin{equation}\label{eqn:ourgen}
\tau_0 + %
\lambda_{c_1^*} \delta_0 + %
\rho_{d_1} \delta^\dagger_0 + %
\lambda_{c_1^*} \rho_{d_1} \pi_0 + %
\lambda_{c_0^*} + \rho_{d_0}.
\end{equation}
This class of semigroups includes both the Lindsay--Sinha and the
Bahn--Park examples, as well as those obtained by unitary conjugation;
the generators of the latter correspond to the case
 \[
c = d = \begin{bmatrix} -\I h - \hlf l^*l \\[1ex] l \end{bmatrix},
\quad \mbox{where } h = h^*.
\]

\subsection{Conventions}
Hilbert spaces are complex with inner products linear in their second
argument. The linear, Hilbert-space and ultraweak tensor products are
denoted by~$\algten$, $\otimes$ and $\uwkten$, respectively. For a
Hilbert space $\hilb$ we adopt the Dirac-inspired notation
$\ket{\hilb}$ for $\bop{\C}{\hilb}$ and $\bra{\hilb}$ for the
topological dual $\bop{\hilb}{\C}$, writing $\ket{u}$ for the operator
$\lambda \mapsto \lambda u$ and $\bra{u}$ for the functional
$v \mapsto \langle u, v\rangle$, where $u \in \hilb$. Recall the
\emph{$E$ notation},
\begin{equation}\label{eqn:Enotation}
E_u := \ket{u} \otimes \id \quad \mbox{and} \quad %
E^u := ( E_u )^* = \bra{u} \otimes \id \qquad ( u \in \hilb ),
\end{equation}
in which $\id$ denotes the identity operator on a Hilbert space
determined by context. The following commutator and anticommutator
notation is also used for elements of an algebra:
\begin{equation}\label{eqn:commutator}
[ a , b ]:= a b - b a \quad \mbox{and} \quad %
\{ a, b \}:= a b + b a.
\end{equation}

\section{Multipliers for quantum stochastic flows}

Fix now, and for the rest of the paper, Hilbert spaces~$\ini$
and~$\mul$, referred to as the \emph{initial space} and
\emph{multiplicity space} or \emph{noise dimension space},
respectively. Fix also a von~Neumann algebra~$\vNa$ acting faithfully
on $\ini$. Set $\widehat{\mul}:= \C \oplus \mul$,
\begin{equation}\label{eqn:chat}
\widehat{c} := \begin{pmatrix} \,1\, \\ \,c\, \end{pmatrix} \in \mmul %
\quad ( c \in \mul ) \qquad \mbox{ and } \qquad
\vac := \widehat{0} = \begin{pmatrix} \,1\, \\ \,0\, \end{pmatrix}.
\end{equation}
Our basic reference for quantum stochastic calculus is~\cite{Lin05}.

For a subinterval $J$ of $\R_+$, let $\fock_J$ denote the Boson Fock
space over $L^2( J; \mul )$ and let $\vNnoise_J := \bop{\fock_J}{}$.
For brevity, set $\fock := \fock_{\R_+}$,
$\fock_{t)} := \fock_{[ 0, t )}$ and
$\fock_{[t} := \fock_{[ t, \infty )}$, with corresponding
abbreviations for the noise algebra $\vNnoise = \bop{\fock}{}$. The
identifications
\[
\fock = \fock_{s)} \otimes \fock_{[s} = %
\fock_{s)} \otimes \fock_{[ s, t )} \otimes \fock_{[t} %
\qquad ( 0 \le s \le t < \infty ),
\]
which arise from the exponential property of Fock space, entail the
identifications
\[
\vNnoise = \vNnoise_{s)} \uwkten \vNnoise_{[s} = %
\vNnoise_{s)} \uwkten \vNnoise_{[ s, t )} \uwkten \vNnoise_{[t} %
\qquad ( 0 \le s \le t < \infty ).
\]
The notation $\Vac_J$, $\id_J$ and $\idop_J$ for the vacuum vector in
$\fock_J$, the identity operator on $\fock_J$ and the identity map on
$\vNnoise_J$, respectively, is also useful, with corresponding
abbreviations for other intervals, such as $\Vac_{[s}$, $\id_{[s}$ and
$\idop_{[s}$, as above.

Denote by $\Delta$ any of the following projections:
\begin{equation}\label{eqn:abuse}
P_\mul \in \bop{\mmul}{}, \quad %
P_\mul \otimes \idop_\vNa \in \bop{\mmul}{} \uwkten \vNa %
\quad \mbox{and} \quad %
P_\mul \otimes \idop_\vNa \otimes \id_\fock \in %
\bop{\mmul}{} \uwkten \vNa \uwkten \vNnoise,
\end{equation}
where
$P_\mul = %
\left[\begin{smallmatrix}
 0 & 0 \\[0.5ex]
 0 & \id_\mul
\end{smallmatrix}\right] \in \bop{\mmul}{}$
is the orthogonal projection onto $\mul$.

The right shift
\[
s_t : \elltwo \to L^2( [ t, \infty ) ; \mul ); \ %
f \mapsto f( \, \cdot - t ) \qquad ( t \ge 0 )
\]
has second quantisation
\[
S_t : \fock \to \fock_{[t}; \ \evec{f} \mapsto \evec{s_t f},
\]
where $\evec{g}$ denotes the exponential vector corresponding to the
vector~$g$, and the map
\[
\sigma_t : \vNa \uwkten \vNnoise{} \to %
\vNa \uwkten \vNnoise_{[t}{}; \ %
T \mapsto ( \id_\ini \otimes S_t ) T ( \id_\ini \otimes S_t )^*
\]
is a normal $*$-isomorphism for all $t \ge 0$.

\begin{definition}\label{def:cocycle}
A \emph{vacuum-adapted quantum stochastic cocycle} $k$ on $\vNa$ is a
family of normal completely bounded maps
$( k_t :\vNa \to \vNa \uwkten \vNnoise{} )_{t \ge 0}$ such
that, for all $a \in \vNa$ and $s$, $t \ge 0$,
\begin{mylist}
\item[($\Omega$-C i)] $k_0( a ) = a \otimes \ket{\Vac}\bra{\Vac}$,
\item[($\Omega$-C ii)]
$k_t( a ) = k_{t)}( a ) \otimes \ket{\Vac_{[t}}\bra{\Vac_{[t}}$,
where $k_{t)}( a ) \in \vNa \uwkten \vNnoise_{t)}{}$,
\item[(C iii)]
$k_{s + t} = \widehat{k}_s \comp \sigma_s \comp k_t$,
where $\widehat{k}_s := k_{s)} \uwkten \idop_{[s}$
\item[and (C iv)] $r \mapsto k_r( a )$ is ultraweakly continuous.
\end{mylist}
Such a family is a \emph{flow} on $\vNa$ if each $k_{t)}$ is
 $*$-homomorphic and unital. Following tradition we use the letter $j$
for quantum stochastic flows.
\end{definition}

In the standard theory, ($\Omega$-C~i) and ($\Omega$-C~ii) are
replaced by their identity-adapted counterparts,
\begin{mylist}
\item[($I$-C i)] $k_0( a ) = a \otimes \id_\fock$
\item[and ($I$-C ii)] $k_t( a ) = k_{t)}( a ) \otimes \id_{[t}$,
where $k_{t)}( a ) \in \vNa \uwkten \vNnoise_{t)}{}$.
\end{mylist}

\begin{remark}
The prescription
\begin{equation}\label{eqn:bijective}
k^{(\Vac)} = \bigl( %
k_{t)}( \, \cdot \, ) \otimes \ket{\Vac_{[t}}\bra{\Vac_{[t}} %
\bigr)_{t\ge 0} \ \mapsto \ %
k^{(I)} = \bigl( %
k_{t)}( \, \cdot \, ) \otimes \id_{[t} %
\bigr)_{t\ge 0}
\end{equation}
gives a bijective correspondence between the class of vacuum-adapted
quantum stochastic cocycles and the class of identity-adapted
quantum stochastic cocycles. Note that
\begin{equation}
k_{t)}( a ) = E^{\Vac_{[t}} k_t( a ) E_{\Vac_{[t}} %
\qquad ( t \ge 0, \, a \in \vNa )
\end{equation}
in both cases.
\end{remark}

In terms of the orthogonal projection
\begin{equation}\label{eqn:vacproj}
P_t := %
\id_\ini \otimes I_{t)} \otimes \ket{\Vac_{[t}}\bra{\Vac_{[t}},
\end{equation}
condition ($\Vac$-C~ii) becomes
\[
k_t( a ) = P_t k_t( a ) P_t,
\]
whereas ($I$-C~ii) only implies the weaker commutation relation
\[
k_t( a ) P_t = P_t k_t( a ).
\]

Let
\[
\expn := %
\idop_\vNa \uwkten \omega_\Vac : \vNa \uwkten \vNnoise{} \to \vNa
\]
denote the \emph{vacuum expectation}, where $\omega_\Vac$ is the state
on $\vNnoise$ corresponding to the vacuum vector~$\Vac$.

\begin{proposition}\label{prp:semigroup}
Let $k$ be a vacuum-adapted quantum stochastic cocycle on $\vNa$.
The ultraweakly continuous family of normal completely bounded
maps $( \expn \circ k_t )_{t \ge 0}$ on $\vNa$ forms a semigroup,
called the \emph{vacuum-expectation semigroup} of $k$.
\end{proposition}
\begin{proof}
For all $t \ge 0$, the conditional expectation
\begin{equation}\label{eqn:condexp}
\expn^\Vac_t : \vNa \uwkten \vNnoise{} \to \vNa \uwkten \vNnoise{};%
\ T \mapsto \bigl( \idop_{\vNa \uwkten \vNnoise_{t)}{}} \uwkten %
\omega_{\Vac_{[t}} \bigr)( T ) \otimes \ket{\Vac_{[t}}\bra{\Vac_{[t}}
= P_t T P_t
\end{equation}
has the tower property $\expn \comp \expn^\Vac_t = \expn$. The claim
follows since any vacuum-adapted quantum stochastic cocycle satisfies
the identity
 \[
\expn^\Vac_t \circ \widehat{k}_t \circ \sigma_t = k_t \circ \expn %
\qquad ( t \ge 0 ).\qedhere
\]
\end{proof}

Quantum stochastic differential equations of the following form are a
basic source of quantum stochastic cocycles.

\begin{remark}
Under the correspondence~\eqref{eqn:bijective}, $k^{(\Vac)}$ satisfies
a quantum stochastic differential equation of the form
\begin{equation}\label{eqn:QSDEVac}
k_0( a ) = a \otimes \ket{\Vac}\bra{\Vac}, \qquad %
\rd k_t = \widetilde{k}_t\bigl( \psi( a ) \bigr) \intd \Lambda_t
\end{equation}
on a subset $\vNa_0$ of $\vNa$, where
$\widetilde{k}_t := \idop_{\bop{\mmul}{}} \uwkten k_t$, if and only if
$k^{(I)}$ satisfies a quantum stochastic differential equation of the
form
\begin{equation}\label{eqn:QSDE}
k_0( a ) = a \otimes \id_\fock, \qquad %
\rd k_t = \widetilde{k}_t\bigl( \phi( a ) \bigr) \intd \Lambda_t
\end{equation}
on $\vNa_0$, where the maps
$\psi$, $\phi: \vNa_0 \to \bop{\mmul}{}\uwkten \vNa$ are related by
the following identity:
\[
\psi( a ) = \phi( a ) + \Delta \otimes a \qquad ( a \in \vNa_0 ).
\]
This is proved in \cite[Theorem~7.3]{Blt10}. Here $\Lambda$ is the
matrix of fundamental quantum stochastic integrators \cite{HP84};
see \cite{Lin05}.
\end{remark}

\begin{remark}[{\cite[Section 6]{LiW00a}}]
Let the map $\phi : \vNa \to \bop{\mmul}{} \uwkten \vNa$ have the
block-matrix form
\begin{equation}\label{eqn:cocmatrix}
\phi( a ) = %
\begin{bmatrix}
 \I [ h, a ] - \hlf \{ r^* r, a \} + r^* \pi( a ) r & %
 a r^* - r^* \pi( a ) \\[1ex]
 r a - \pi( a ) r & \pi( a ) - \id_\mul \otimes a
\end{bmatrix}
\qquad ( a \in \vNa ),
\end{equation}
where $h \in \vNa$ is self adjoint, $r \in \ket{\mul} \uwkten \vNa$
and $\pi : \vNa \to \bop{\mul}{} \uwkten \vNa$ is a normal unital
$*$-homomorphism. Then the quantum stochastic differential
equation~\eqref{eqn:QSDE} has a unique solution and this is an
identity-adapted quantum stochastic flow. Conversely, if an
identity-adapted quantum stochastic flow satisfies~\eqref{eqn:QSDE}
for some normal bounded map
$\phi : \vNa \to \bop{\mmul}{} \uwkten \vNa$ then $\phi$ has the
form~\eqref{eqn:cocmatrix}.
\end{remark}

\begin{definition}\label{def:multiplier}
Let $j$ be a vacuum-adapted quantum stochastic flow on $\vNa$. A
family of operators
$M = ( M_t )_{t \ge 0 }$ in $\vNa \uwkten \vNnoise{}$ is a
\emph{multiplier} for $j$ if, for all $s$, $t \ge 0$,
\begin{mylist}
\item[(M i)] $M_0 = \id_{\ini \otimes \fock}$,
\item[(M ii)] $M_t P_t = P_t M_t$,
\item[(M iii)] $M_{s + t} = J_s( M_t ) M_s$, where
$J_s:= \widehat{\jmath}_s \comp \sigma_s$
\item[and (M iv)] $r \mapsto M_r$ is strongly continuous.
\end{mylist}
The Banach--Steinhaus Theorem and condition~(M iv) imply that $M$ is
locally bounded.
\end{definition}

\begin{theorem}[{\Cf~\cite[Theorem~2.1]{BaP03}}]%
\label{thm:semigroup}
Let $M$ and $N$ be multipliers for the vacuum-adapted quantum
stochastic flow $j$. The ultraweakly continuous normal completely
bounded family
\[
\semigroup := %
\bigl( a \mapsto \expn[ M_t^* j_t( a ) N_t ] \bigr)_{t \ge 0}
\]
forms a semigroup, which is completely contractive if $M$ and $N$ are
contractive and is completely positive if $M = N$.
\end{theorem}
\begin{proof}
To prove the semigroup property, let $a \in \vNa$ and $s$,~$t \ge 0$.
By the tower property for the conditional expectation $\expn^\Vac_s$
defined in \eqref{eqn:condexp}, it follows that
\begin{align}
\semigroup_{s + t}( a ) & = %
\expn\bigl[ \expn^\Vac_s[ M_s^* J_s( M_t^* ) J_s ( j_t( a ) ) %
J_s( N_t ) N_s ] \bigr] \tag*{by (C~iii) and (M~iii)} \\[1ex]
 & = \expn[ M_s^* \expn^\Vac_s[ J_s( M_t^* j_t( a ) N_t ) ] N_s ] %
\tag*{by (M~ii)} \\[1ex]
 & = \expn\bigl[ M_s^* j_s( \expn[ M_t^* j_t( a ) N_t ] ) %
N_s \bigr] \label{eqn:antepenultimate} \\[1ex]
 & = \semigroup_s( \semigroup_t ( a ) ). \nonumber
\end{align}
For the equality \eqref{eqn:antepenultimate}, note that if
$a \in \vNa$ and $b \in \vNnoise$ then
\[
\expn^\Vac_s[ J_s( a \otimes b ) ) \bigr] = %
\langle \Vac, b \Vac \rangle \, j_s( a ) = %
j_s( \expn[ a \otimes b ] );
\]
thus $\expn^\Vac_s \circ J_s = j_s \circ \expn$, by linearity and
ultraweak continuity.
\end{proof}

\begin{remark}
Some of the ideas in this section go back to early work of Accardi
\cite[Sections~2 and~4]{Acc78}; see also \cite[Section~2.3]{AFL82}.
\end{remark}

\section{A vacuum-adapted quantum stochastic differential equation}

Let $( u_t )_{t \in \R}$ be a strongly continuous one-parameter
unitary group in $\vNa$, let $B = ( B_t )_{t \ge 0}$ be the canonical
Brownian motion on Wiener's probability space $\Wiener$ and, taking
$\mul = \C$, identify $L^2( \Wiener )$ with $\fock$ \via the
Wiener--It\^o--Segal isomorphism. If the unitary operator
$U_t \in \vNa \uwkten \vNnoise$ is such that
\[
U_t \xi : \omega \mapsto u_{B_t( \omega )} \xi( \omega ) = %
u_{\omega( t )} \xi( \omega ) %
\qquad ( \xi \in L^2( \Wiener; \ini ) )
\]
then the family of maps
$\bigl( j^B_t : %
a \mapsto U_t ( a \otimes \id_\fock ) U_t^* \bigr)_{t \ge 0}$ is an
identity-adapted quantum stochastic flow on $\vNa$
(\cite[Lemma~3.1]{BaP03}, \cf~\cite[Section~5]{Lin05}).

Bahn and Park considered the operator stochastic differential equation
\begin{equation}\label{eqn:sde}
M^a_0 = \id_{\ini \otimes \fock}, \qquad %
\rd M^a_t = j^B_t( a ) P_t M^a_t \intd B_t - %
\hlf j^B_t( a^2 ) P_t M^a_t \intd t,
\end{equation}
where $a \in \vNa$, and obtained a solution pointwise in
$L^2( \Wiener; \ini )$ \cite[Proposition~3.2]{BaP03}. They showed
that the collection of operators $( M^a_t )_{t \ge 0}$ forms a
multiplier for the quantum stochastic flow~$j^B$
\cite[Proposition~3.3]{BaP03}.

Fix $a \in \vNa$ and set $N_t := M^a_t - \id_{\ini \otimes \fock}$ for
all $t \ge 0$, so that
\[
N_t \xi = %
\int_0^t Q_s \xi \intd B_s - \frac12 \int_0^t R_s \xi \intd s + %
\int_0^t Q_s N_s \xi \intd B_s - \frac12 \int_0^t R_s N_s \xi \intd s
\]
for all $\xi \in L^2( \Wiener; \ini )$, where
\[
Q_t := j^B_t( a ) P_t \qquad \mbox{and} \qquad %
R_t := j^B_t( a^2 ) P_t.
\]
As $\bigl( j^B_t( b ) \bigr)_{t \ge 0}$ is identity adapted for all
$b \in \vNa$, the processes $Q$ and~$R$ are vacuum adapted. By
\cite[Theorem~2.2]{Blt07}, the process $N$ above is the unique
vacuum-adapted solution of the quantum stochastic differential
equation
\begin{equation}\label{eqn:bpqsde}
N_0 = 0, \qquad %
\rd N_t = Q_t \intd A^\dagger_t - \hlf R_t \intd t + %
Q_t N_t \intd A^\dagger_t - \hlf R_t N_t \intd t.
\end{equation}

To see that~\eqref{eqn:bpqsde} is the correct quantum stochastic
generalisation of~\eqref{eqn:sde}, for simplicity take $\ini = \C$ and
let~$\stex{f}$ denote the Brownian exponential corresponding to
$f \in L^2( \R_+ )$, \ie the unique element of $L^2( \Wiener )$ such
that
\[
\stex{f}_t := \expn_\Wiener[\stex{f} | \filt_t ] = 1 + %
\int_0^t f( s ) \expn_\Wiener[\stex{f} | \filt_s ] %
\intd B_s \qquad ( t \ge 0 ),
\]
where $( \filt_t )_{t \ge 0}$ is the canonical filtration generated by
the Brownian motion~$B$. (Recall that~$\stex{f}$ corresponds
to~$\evec{f}$ and $\expn_\Wiener[ \, \cdot \, | \filt_t ]$ to~$P_t$.)
If~$( X_t )_{t \ge 0}$ is a process of bounded operators on $\fock$
with locally bounded norm and such that $X_t P_t = P_t X_t$ for all
$t \ge 0$ then, by the (classical) It\^o product formula,
\begin{align*}
\expn_\Wiener\Bigl[ \overline{\stex{f}} %
\int_0^t X_s P_s \stex{g} \intd B_s \Bigr] %
& = \expn_\Wiener\Bigl[ \int_0^t \overline{f( s ) \stex{f}_s} %
 X_s \stex{g}_s \intd s \Bigr] \\
 & = \Bigl\langle \evec{f}, %
\int_0^t X_s P_s \intd A^\dagger_s \evec{g} \Bigr\rangle %
\qquad ( f, g \in L^2( \R_+ ) ).
\end{align*}

\begin{definition}
For a Hilbert space $\hilb$, a \emph{bounded process in}
$\bop{\hilb}{} \uwkten \vNa$ is a family of operators
$Z = ( Z_t )_{t \ge 0}$ in
$\bop{\hilb}{} \uwkten \vNa \uwkten \vNnoise$ such that
\[
t \mapsto \langle \zeta', Z_t \zeta \rangle \mbox{ is measurable }
\qquad (\zeta, \zeta' \in \hilb \otimes \ini \otimes \fock);
\]
such a process is \emph{vacuum adapted} if
\[
Z_t = ( \id_\hilb \otimes P_t ) Z_t ( \id_\hilb \otimes P_t ) %
\qquad ( t \ge 0 )
\]
or, equivalently,
\[
Z_t = Z_{t)} \otimes \ket{\Vac_{[t}}\bra{\Vac_{[t}} %
\quad \mbox{for some } %
Z_{t)} \in \bop{\hilb}{} \uwkten \vNa \uwkten \vNnoise_{[ 0, t )} %
\quad ( t \ge 0 ).
\]
A vacuum-adapted bounded process $G$ in $\bop{\mmul}{} \uwkten \vNa$
is an \emph{integrand} process if its block-matrix form
$\left[ \begin{smallmatrix} k & m \\[0.5ex]
 l & n \end{smallmatrix} \right]$
is such that
\[
\| G \|_t := \| k \|_{1, t} + %
\| l \|_{2, t} + %
\| m \|_{2, t} + %
\| n \|_{\infty, t} < \infty
 \quad
 (t \ge 0),
\]
where, for $p = 1, 2$ or $\infty$, $\| f \|_{p, t}$ denotes the
$L^p$~norm of the function $1_{[ 0, t )} f$.
\end{definition}

The following result is the coordinate-independent version of
\cite[Proposition~37]{Blt04}, with non-trivial initial space. Recall
the notation~\eqref{eqn:Enotation} and~\eqref{eqn:chat}.

\begin{proposition}
Let $G$ be an integrand process. There is a unique bounded
vacuum-adapted process
$\int G \intd \Lambda = %
\bigl( \int_0^t G_s \intd \Lambda_s \bigr)_{t \ge 0}$ in $\vNa$
such that
\[
\langle u \evec{f}, %
\int_0^t G_s \intd \Lambda_s v \evec{g} \rangle = %
\int_0^t \langle u \evec{f}, %
E^{\widehat{f( s )}} G_s E_{\widehat{g( s )}} v \evec{g} %
\rangle \intd s \qquad ( t \ge 0 )
\]
for all $u$,~$v \in \ini$ and $f$,~$g \in \elltwo$. Moreover, the
following inequality holds:
 \[
\| \int_0^t G_s \intd \Lambda_s \| \le \| G \|_t %
\qquad ( t \ge 0 ).
\]
\end{proposition}

We shall need to pass suitably adapted operators inside quantum
stochastic integrals. The next lemma takes care of this.

\begin{lemma}\label{lem:inside}
Let $G$ be an integrand process such that $G \Delta \equiv 0$ and let
$X$ be a bounded vacuum-adapted process in $\vNa$. Then
\begin{equation}
\int_s^t G_r \intd \Lambda_r \, X_s = %
\int_s^t G_r ( \id_\mmul \otimes X_s ) \intd \Lambda_r %
\qquad ( 0 \le s \le t ).
\end{equation}
\end{lemma}
\begin{proof}
Let $u$,~$v \in \ini$ and $f$,~$g \in \elltwo$; note that
\[
\langle u \evec{f}, %
\int_0^t G_r \intd \Lambda_r \, v \evec{g} \rangle = %
\int_0^t \langle u \evec{f}, %
E^{\widehat{f( r )}} G_r E_\vac v \evec{g} \rangle \intd r,
\]
since $\Delta^\perp E_{\widehat{c}} = E_\vac$ for all $c \in \mul$. If
$A \in \bop{\ini \otimes \fock_{s)}}{}$ and
$\xi \in \ini \otimes \fock$ then, setting
$P_{[s}:= \ket{\Vac_{[s}}\bra{\Vac_{[s}}$ for brevity, it follows that
\begin{align*}
\langle u \evec{f}, \int_s^t G_r \intd \Lambda_r %
( A \otimes P_{[s} ) \xi \rangle & = %
\int_s^t \langle u \evec{f}, E^{\widehat{f( r )}} G_r E_\vac %
( A \otimes P_{[s} ) \xi \rangle \intd r \\[1ex]
 & = \int_s^t \langle u \evec{f}, E^{\widehat{f( r )}} G_r %
( \id_\mmul \otimes A \otimes P_{[s} ) %
E_\vac \xi \rangle \intd r \\[1ex]
 & = \langle u \evec{f}, \int_s^t G_s %
( \id_\mmul \otimes A \otimes P_{[s} ) %
\intd \Lambda_r \xi \rangle. \qedhere
\end{align*}
\end{proof}

The following existence and uniqueness theorem is sufficiently general
for present purposes.

\begin{theorem}\label{thm:qsde}
Let $G$ and $X$ be as in Lemma~\ref{lem:inside}, with $X$ locally
bounded in norm. Then there is a unique vacuum-adapted process~$Z$ in
$\vNa$ such that
\begin{equation}\label{eqn:qsde}
Z_t = X_t + \int_0^t G_s ( \id_\mmul \otimes Z_s ) \intd \Lambda_s %
\qquad ( t \ge 0 ).
\end{equation}
Furthermore,
\[
\| Z \|_{\infty, t} \le \sqrt{2} \, \| X \|_{\infty, t} %
\exp( 2 \| l \|_{2, t}^2 + 2 \| k \|_{1, t}^2 ) \qquad (t \ge 0),
\]
where
$\left[\begin{smallmatrix} k & 0 \\[0.5ex]
 l & 0 \end{smallmatrix}\right]$
is the block-matrix form of $G$, and $Z$ is norm continuous if and
only if $X$ is.
\end{theorem}
\begin{proof}
Define a sequence of processes $( X^{(n)} )_{n\ge 0}$ inductively by
letting $X^{(0)} := X$ and
\[
X^{(n + 1)}_t := %
\int_0^t G_s ( \id_\mmul \otimes X^{(n)}_s ) \intd \Lambda_s %
\qquad ( t \ge 0 ).
\]
This process is well defined and such that
\[
\| X^{(n + 1)}_t \| \le %
\| k \, X^{(n)} \|_{1, t} + \| l \, X^{(n)} \|_{2, t} %
\qquad ( t \ge 0 ),
\]
so, integrating by parts,
\[
\| X^{(n + 1)} \|_{\infty, t}^2 \le %
2 \| k \, X^{(n)} \|_{1, t}^2 + 2 \| l \, X^{(n)} \|_{2, t}^2 \le %
\int_0^t c( s ) \| X^{(n)} \|_{\infty, s}^2 \intd s,
\]
where
\[
c( s ) := 4 \| k_s \| \int_0^s \| k_r \| \intd r + 2 \| l_s \|^2.
\]
It follows that
\[
\| X^{(n + 1)} \|^2_{\infty, t} \le \frac{1}{n!} %
\Bigl( \int_0^t c( s ) \intd s \Bigr)^n \| X \|_{\infty, t}^2 %
\qquad ( n \ge 0, \, t \ge 0 ),
\]
so $Z_t := \sum_{n = 0}^\infty X^{(n)}_t$ exists for all $t \ge 0$,
the series being convergent in norm. A dominated-convergence
argument shows that $Z$ satisfies~\eqref{eqn:qsde} and, since
\[
\| Z_t \|^2 \le %
2 \| X_t \|^2 + 2 \int_0^t c( s ) \| Z_s \|^2 \intd s %
\qquad ( t \ge 0 ),
\]
the inequality and so uniqueness follow from Gronwall's lemma. The
final claim is immediate.
\end{proof}

\section{Multipliers \via quantum stochastic differential equations}

Fix a vacuum-adapted quantum stochastic flow $j$ on $\vNa$ and let
\[
J_t := \widehat{\jmath}_t \circ \sigma_t : %
\vNa \uwkten \vNnoise \to %
\vNa \uwkten \vNnoise_{t)} \uwkten \vNnoise_{[t} = %
\vNa \uwkten \vNnoise
\]
and $\widetilde{J}_t := \idop_{\bop{\mmul}{}} \uwkten J_t$,
for all~$t \ge 0$. The ultraweakly continuous family of normal unital
$*$-homomorphisms $( J_t )_{t\ge 0}$ form a semigroup
(\cf~\cite[Proposition~4.3]{LiW00b}).

The following result is a vacuum-adapted version
of~\cite[Lemma~5.1]{BLS11} which suffices here.

\begin{lemma}\label{prp:shift}
If the integrand process $G$ is norm continuous then the family of
operators
$\bigl( %
\indf{[ s, \infty )}( r ) \widetilde{J}_s( G_{r - s} ) %
\bigr)_{r\ge 0}$, %
where $1_A$ denotes the indicator function of the set~$A$, defines an
integrand process such that
\[
J_s\Bigl( \int_0^t G_r \intd \Lambda_r \Bigr) = %
\int_s^{s + t} \widetilde{J}_s( G_{r - s} ) \intd \Lambda_r %
\qquad ( t \ge 0 ).
\]
\end{lemma}
\begin{proof}[Sketch proof]
Apply the ampliation of the vector functional
$A \mapsto \langle \evec{f}, A \evec{g} \rangle$ to the left-hand
side, then consider suitable Riemann sums.
\end{proof}

With this technical lemma we can construct multipliers of $j$ by
solving quantum stochastic differential equations with coefficients
driven by $j$.

\begin{lemma}\label{lem:exist}
For all $c \in \ket{\mmul} \uwkten \vNa$ there is a unique process
$M^c = (M^c_t)_{t \ge 0}$ in $\vNa$ such that
$M^c - I = ( M^c_t - \id_{\ini \otimes \fock} )_{t \ge 0}$ is vacuum
adapted and
\[
 M^c_t = \id_{\ini \otimes \fock} + %
\int_0^t \widetilde{\jmath}_s( c E^\vac ) %
( \id_\mmul \otimes M^c_s ) \intd \Lambda_s,
\]
where $\widetilde{\jmath}_s := \idop_{\bop{\mmul}{}} \uwkten j_s$
and
$\vac := %
\left( \begin{smallmatrix} 1 \\[0.5ex] 0 \end{smallmatrix} \right) %
\in \mmul$, \ie
\[
\langle u \evec{f}, ( M^c - \id )_t v \evec{g} \rangle = %
\int_0^t \langle u \evec{f}, %
j_s( E^{\widehat{f( s )}} c) M^c_ s v \evec{g} \rangle \intd s %
\qquad ( t \ge 0 )
\]
for all $u$,~$v \in \ini$ and $f$,~$g \in \elltwo$. The process $M^c$
is norm continuous.
\end{lemma}
\begin{proof}
Define an integrand process $G$ by setting
$G_t := \widetilde{\jmath}_t( c E^\vac )$ for all $t \ge 0$. In view
of the identity
$\widetilde{\jmath}( \, \cdot \, ) \Delta = %
 \widetilde{\jmath}( \, \cdot \, \Delta )$,
which exploits the abuse of notation~\eqref{eqn:abuse}, and the fact
that $E^\vac \Delta = 0$, Theorem~\ref{thm:qsde} gives a
vacuum-adapted process $N$ in $\vNa$ which is norm continuous and such
that
\begin{equation}\label{eqn:hdef}
N_t = \int_0^t G_s \intd \Lambda_s + %
\int_0^t G_s ( \id_\mmul \otimes N_s ) \intd \Lambda_s %
\qquad ( t \ge 0 ).
\end{equation}
Hence $M^c_t := \id_{\ini \otimes \fock} + N_t$ is a norm-continuous
process as required; uniqueness holds because the solution
of~\eqref{eqn:hdef} is unique.
\end{proof}

\begin{theorem}\label{thm:mult}
For all $c \in \ket{\widehat{\mul}} \uwkten \vNa$ the process $M^c$
given by Lemma~\ref{lem:exist} is a multiplier for $j$.
\end{theorem}
\begin{proof}
It suffices to verify that condition (M~iii) of
Definition~\ref{def:multiplier} holds. Fix~$s \ge 0$ and let
\[
M_t := \left\{ \begin{array}{ll}
 M^c_t & \mbox{ if } t \in [ 0, s ), \\[1ex]
 J_s( M^c_{t - s} ) M^c_s & \mbox{ if } t \in [ s, \infty ).
\end{array}\right.
\]
Now
$\widetilde{J}_s \comp \widetilde{\jmath}_{r - s} = %
\widetilde{\jmath}_r$
for all $r \ge s$, by (C~iii) of Definition~\ref{def:cocycle}, so
Lemma~\ref{lem:inside} and Proposition~\ref{prp:shift} imply that
\begin{align*}
M_{s + t} & = %
 J_s\Bigl( \id_{\ini \otimes \fock} + %
 \int_0^t \widetilde{\jmath}_r( c E^\vac ) %
 ( \id_\mmul \otimes M^c_r ) \intd \Lambda_r \Bigr) M^c_s \\[1ex]
 & = M^c_s + \int_s^{s + t} \widetilde{J}_s\bigl( %
 \widetilde{\jmath}_{r - s}( c E^\vac ) %
 ( \id_\mmul \otimes M^c_{r - s} ) \bigr) %
 ( \id_\mmul \otimes M^c_s ) \intd \Lambda_r \\[1ex]
 & = M^c_s + \int_s^{s + t} \widetilde{\jmath}_r( c E^\vac ) %
 \bigl( \id_\mmul \otimes ( J_s( M^c_{r - s} ) M^c_s ) \bigr) %
 \intd \Lambda_r \\[1ex]
 & = \id_{\ini \otimes \fock} + %
 \int_0^s \widetilde{\jmath}_r( c E^\vac ) %
 ( \id_\mmul \otimes M^c_r ) \intd \Lambda_r + %
 \int_s^{s + t} \widetilde{\jmath}_r( c E^\vac ) %
 ( \id_\mmul \otimes M_r ) \intd \Lambda_r \\[1ex]
 & = \id_{\ini \otimes \fock} + %
 \int_0^{s + t} \widetilde{\jmath}_r( c E^\vac ) %
 ( \id_\mmul \otimes M_r ) \intd \Lambda_r \qquad ( t \ge 0 ).
\end{align*}
By Lemma~\ref{lem:exist}, $M \equiv M^c$ and
$M^c_{t + s} = M_{t + s} = J_s( M^c_t ) M^c_s$, as required.
\end{proof}

\section{Semigroup perturbation}

For vacuum-adapted integrands the quantum It\^o product formula takes
the following form \cite[Section~5.4]{Blt04}.

\begin{lemma}\label{lem:qif}
Let $Z := \int G \intd \Lambda$ and
$Z' := \int G' \intd \Lambda$ for integrand processes $G$
and~$G'$. Then
\[
H := ( \id_\mmul \otimes Z ) \Delta^\perp G' + %
G \Delta^\perp ( \id_\mmul \otimes Z' ) + G \Delta G'
\]
defines an integrand process such that
$Z Z' = \int H \intd \Lambda$.
\end{lemma}

The product of three integrals gives the following.

\begin{corollary}\label{cor:qif}
Let $G$, $G'$ and $G''$ be integrand processes and let
$Z := \int G \intd \Lambda$, $Z' := \int G' \intd \Lambda$
and $Z'' := \int G'' \intd \Lambda$. Then
\begin{multline*}
H := ( \id_\mmul \otimes Z Z' ) \Delta^\perp G'' + %
( \id_\mmul \otimes Z ) \Delta^\perp G' \Delta^\perp %
( \id_\mmul \otimes Z'' ) + %
 G \Delta^\perp ( \id_\mmul \otimes Z' Z'' ) \\
 + ( \id_\mmul \otimes Z ) \Delta^\perp G' \Delta G'' + %
 G \Delta G' \Delta^\perp ( \id_\mmul \otimes Z'' ) + %
 G \Delta G' \Delta G''
\end{multline*}
is an integrand process such that
$Z Z' Z'' = \int H \intd \Lambda$.
\end{corollary}

We may now give the main result.

\begin{theorem}\label{thm:main}
Let $\psi : \vNa_0 \to \vNa \uwkten \bop{\mmul}{}$, where $\vNa_0$ is
a subset of $\vNa$, and suppose~$j$ satisfies the vacuum-adapted
quantum stochastic differential equation
\[
j_0( x ) = x \otimes \ket{\Vac}\bra{\Vac}, \qquad %
\rd j_t( x ) = \widetilde{\jmath}_t\bigl( \psi( x ) \bigr) %
\intd \Lambda_t \qquad ( x \in \vNa_0 ).
\]
For each $c$,~$d \in \ket{\mmul} \uwkten \vNa$, the generator
$\tau$ of the pointwise ultraweakly continuous semigroup
$\semigroup := \bigl( %
\expn[ ( M^c_t )^* \, j_t( \, \cdot \, ) \, M^d_t ] \bigr)_{t\ge 0}$
satisfies $\dom \tau \supset \vNa_0$ and, for all $x \in \vNa_0$,
\begin{equation}\label{eqn:gen}
\tau( x ) = E^\vac \psi( x ) E_\vac + %
c^* \Delta \psi( x ) E_\vac + E^\vac \psi( x ) \Delta d + %
c^* \Delta \psi( x ) \Delta d + c^* E_\vac x + x E^\vac d.
\end{equation}
\end{theorem}
\begin{proof}
Let $x \in \vNa_0$ and $t \ge 0$; note that
$(M^c_t)^* j_t( x ) M^d_t - j_t( x )$ equals
\begin{multline}
( M^c - \id )^*_t ( j_t - j_0 )( x ) ( M^d - \id )_t + %
( M^c - \id )^*_t ( j_t - j_0 )( x ) \\
+ ( j_t - j_0 )( x ) ( M^d - \id )_t + %
( M^c - \id )^*_t j_0( x ) ( M^d - \id )_t \\
+ ( M^c - \id )^*_t j_0( x ) + j_0( x ) ( M^d - \id )_t.
\label{eqn:b}
\end{multline}
If $u$,~$v \in \ini$ and $f$,~$g \in \elltwo$ then, writing $P_\Vac$
for $\ket{\Vac}\bra{\Vac} \in \vNnoise$,
\begin{align*}
\langle u \evec{f}, j_0( x ) %
( M^d - \id )_t v \evec{g} \rangle & = %
\langle ( x^* u ) \Vac, \int_0^t \widetilde{\jmath}_s( d E^\vac ) %
( \id_\mmul \otimes M^d_s ) \intd \Lambda_s v \evec{g} \rangle \\[1ex]
& = \int_0^t \langle ( x^* u ) \Vac, E^\vac %
\widetilde{\jmath}_s( d E^\vac ) ( \id_\mmul \otimes M^d_s ) %
E_{\widehat{g( s )}} v \evec{g} \rangle \intd s \\[1ex]
 & = \smash[b]{\int_0^t \langle u \evec{f}, ( x \otimes P_\Vac ) %
j_s( E^\vac d ) M^d_s v \evec{g} \rangle \intd s,}
\end{align*}
therefore
\begin{align*}
j_0( x ) ( M^d - \id )_t & = \smash[t]{\int_0^t} %
( x \otimes P_\Vac ) j_s( E^\vac d ) M^d_s \intd s, \\[1ex]
( M^c - \id )^*_t j_0( x ) & = \smash[b]{\int_0^t} %
( M^c_s )^* j_s( c^* E_\vac ) ( x \otimes P_\Vac ) \intd s \\
\intertext{and}
( M^c - \id )^*_t j_0( x ) ( M^d - \id )_t & = %
\smash[t]{\int_0^t} ( M^c - \id )^*_s ( x \otimes P_\Vac ) %
j_s( E^\vac d ) M^d_s \intd s \\
& \qquad + \smash[b]{\int_0^t} ( M^c_s )^* j_s( c^* E_\vac ) %
( x \otimes P_\Vac ) ( M^d - \id )_s \intd s.
\end{align*}
This implies that the sum of the last three terms in~\eqref{eqn:b}
equals
\begin{multline*}
\int_0^t ( M^c_s )^* \bigl( ( x \otimes P_\Vac ) j_s( E^\vac d ) + %
j_s( c^* E_\vac ) ( x \otimes P_\Vac ) \bigr) M^d_s \intd s \\
= \int_0^t ( \widetilde{M}^c_s )^* \bigl( %
( \id_\mmul \otimes x \otimes P_\Vac ) %
\widetilde{\jmath}_s( \Delta^\perp d E^\vac ) + %
\widetilde{\jmath}_s( E_\vac c^* \Delta^\perp ) %
( \id_\mmul \otimes x \otimes P_\Vac ) \bigr) %
\widetilde{M}^d_s \intd \Lambda_s,
\end{multline*}
where
$\widetilde{M}^e_s := \id_\mmul \otimes M^e_s$ for $e = c$, $d$.

After some working, with the aid of Lemma~\ref{lem:qif} and
Corollary~\ref{cor:qif}, it follows that
$( M^c_t )^* j_t( x ) M^d_t - j_0( x )$ equals
\[
\int_0^t \bigl( \widetilde{\jmath}_s( A_1 ) + %
\widetilde{\jmath}_s( A_2 ) \widetilde{M}^d_s + %
(\widetilde{M}^c_s)^* \widetilde{\jmath}_s( A_3 ) + %
(\widetilde{M}^c_s)^* \widetilde{\jmath}_s( A_4 ) %
\widetilde{M}^d_s \bigr) \intd \Lambda_s,
\]
where
\begin{align*}
A_1 & := \Delta \psi( x ) \Delta, \\[1ex]
A_2 & := \Delta \psi( x ) \Delta^\perp + %
\Delta \psi( x ) \Delta d E^\vac, \\[1ex]
A_3 & := \Delta^\perp \psi( x ) \Delta + %
E_\vac c^* \Delta \psi( x ) \Delta \\[1ex]
\mbox{and} \quad A_4 & := \Delta^\perp \psi( x ) \Delta^\perp + %
E_\vac c^* \Delta \psi( x ) \Delta^\perp + %
\Delta^\perp \psi( x ) \Delta d E^\vac \\
 & \qquad + E_\vac c^* \Delta \psi( x ) \Delta d E^\vac + %
E_\vac c^* \Delta^\perp ( \id_\mmul \otimes x ) + %
( \id_\mmul \otimes x ) \Delta^\perp d E^\vac.
\end{align*}
Hence
\begin{align*}
\langle u, ( \semigroup_t( x ) - x ) v \rangle & = %
\langle u \Vac, %
\bigl( ( M^c_t )^* j_t( x ) M^d_t - j_0( x ) \bigr) v %
\Vac \rangle \\[1ex]
 & = \int_0^t \langle u \Vac, \bigl( j_s( E^\vac A_1 E_\vac ) + %
j_s( E^\vac A_2 E_\vac ) M^d_s \\
 & \hspace{4em} + (M^c_s)^* j_s( E^\vac A_3 E_\vac ) + %
( M^c_s )^* j_s( E^\vac A_4 E_\vac ) M^d_s \bigr) %
v \Vac \rangle \intd s \\[1ex]
 & = \int_0^t \langle u \Vac, (M^c_s)^* %
j_s( E^\vac A_4 E_\vac ) M^d_s v \Vac \rangle \intd s \\[1ex]
 & = \int_0^t \langle u, \semigroup_s( y ) v \rangle \intd s,
\end{align*}
where
\[
y = E^\vac \psi( x ) E_\vac + c^* \Delta \psi( x ) E_\vac + %
E^\vac \psi( x ) \Delta d + c^* \Delta \psi( x ) \Delta d + %
c^* E_\vac x + x E^\vac d,
\]
as required.
\end{proof}

\begin{remark}
In terms of the direct-sum decomposition $\mmul = \C \oplus \mul$, if
\[
\psi = \begin{bmatrix} \tau_0 & \delta_0^\dagger \\[1ex]
 \delta_0 & \pi_0 \end{bmatrix}, \qquad %
c = \begin{bmatrix} k_1 \\ l_1 \end{bmatrix} %
\quad \mbox{and} \quad %
d = \begin{bmatrix} k_2 \\ l_2 \end{bmatrix}
\]
then \eqref{eqn:gen} becomes
\[
\tau( x ) = \tau_0( x ) + l^*_1 \delta_0( x ) + %
\delta_0^\dagger( x ) l_2 + l_1^* \pi_0( x ) l_2 + k_1^* x + x k_2 %
\qquad ( x \in \vNa_0 ).
\]
When $\psi$ is bounded and $\vNa_0 = \vNa$, the map $\delta_0$ is a
bounded $\pi_0$-derivation. Since
$\delta_0( \vNa_0 ) \subset \vNa \uwkten \ket{\mul}$, it follows that
$\delta_0$ is implemented (\cite{ChE79}, see \cite[Chapter~6]{Lin05})
and so
\begin{multline*}
\tau( x ) = \I [ h, x ] - \hlf \{ r^* r, x \} + r^* \pi_0( x ) r \\
+ ( x r^* - r^* \pi_0( x ) ) l_2 + %
l_1^* ( r x - \pi_0( x ) r ) + l_1^* \pi_0( x ) l_2 + %
k_1^* x + x k_2
\end{multline*}
for some $h = h^* \in \vNa$ and
$r \in \ket{\mul} \uwkten \vNa$. Equivalently,
\[
\tau( x ) = d_1^* \pi_0( x ) d_2 + e_1^* x + x e_2,
\]
where $d_i = l_i - r$ and $e_i = k_i + r^* l_i - \hlf r^*r - \I h$
for $i = 1$, $2$.
\end{remark}

\par\bigskip\noindent
{\bf Acknowledgments.} We thank Kalyan Sinha for constructive
remarks at an early stage of this project. This work benefited from
the UK-India Education and Research Initiative grant QP-NCG-QI:
\emph{Quantum Probability, Noncommutative Geometry and Quantum
Information}.

\bibliographystyle{amsplain}

\begin{thebibliography}{99}

\bibitem{Acc78} Accardi, L.:
 On the quantum Feynman--Kac formula,
 {\em Rend.\ Sem.\ Mat.\ Fis.\ Milano}~{\bf 48} (1978) 135--180.

\bibitem{AcF83} Accardi, L.; Frigerio, A.:
 Markovian cocycles,
 {\em Proc.\ Roy.\ Irish Acad.\ Sect.\ A}~{\bf 83} (1983) no.~2,
 251--263.

\bibitem{AFL82} Accardi, L.; Frigerio, A.; Lewis, J. T.:
 Quantum stochastic processes,
 {\em Publ.\ Res.\ Inst.\ Math.\ Sci.}~{\bf 18} (1982) no.~1, 97--133.

\bibitem{AlF87} Alicki, R.; Fannes, M.:
 Dilations of quantum dynamical semigroups with classical Brownian
 motion,
 {\em Comm.\ Math.\ Phys.}~{\bf 108} (1987) no.~3, 353--361.

\bibitem{Arv84} Arveson, W.:
 Ten lectures on operator algebras,
 {\em CBMS Regional Conference Series in Mathematics}
 \textbf{55},
 American Mathematical Society, Providence, 1984.

\bibitem{BaP03} Bahn, C.; Park, Y. M.:
 Feynman--Kac representation and Markov property of semigroups
 generated by noncommutative elliptic operators,
 {\em Infinite Dim.\ Anal.\ Quantum Probab.}~{\bf 6} (2003) no.~1,
 103--121.

\bibitem{Blt01} Belton, A. C. R.:
 Quantum $\Omega$-semimartingales and stochastic evolutions,
 {\em J.\ Funct.\ Anal.}~{\bf 187} (2001) no.~1, 94--109.

\bibitem{Blt04} Belton, A. C. R.:
 An isomorphism of quantum semimartingale algebras,
 {\em Q.\ J.\ Math.}~{\bf 55} (2004) no.~2, 135--165.

\bibitem{Blt07} Belton, A. C. R.:
 Alicki--Fannes and Hudson--Parthasarathy evolution equations,
 in:
 {\em Quantum Probability and Infinite Dimensional Analysis}
 {\bf 20} (2007) 128--133, World Scientific, Singapore.

\bibitem{Blt10} Belton, A. C. R.:
 Random-walk approximation to vacuum cocycles,
 {\em J.\ London Math.\ Soc.~(2)}~{\bf 81} (2010) no.~2, 412--434.

\bibitem{BLS11} Belton, A. C. R.; Lindsay, J. M.; Skalski, A. G.:
 Quantum Feynman--Kac perturbations,
 {\em preprint}, 2011.

\bibitem{ChE79} Christensen, E.; Evans, D. E.:
 Cohomology of operator algebras and quantum dynamical semigroups,
 {\em J.\ London Math.\ Soc.~(2)}~{\bf 20} (1979) no.~2, 358--368.

\bibitem{EvH90} Evans, M. P.; Hudson, R. L.:
 Perturbations of quantum diffusions,
 {\em J.\ London Math.\ Soc.~(2)} {\bf 41} (1990) no.~2, 373--384.

\bibitem{HIP8284} Hudson, R. L.; Ion, P. D. F.; Parthasarathy, K. R.:
 Time-orthogonal unitary dilations and noncommutative Feynman--Kac
 formulae,
 {\em Comm.\ Math.\ Phys.}~{\bf 83} (1982) no.~2, 261--280;
 Time-orthogonal unitary dilations and noncommutative Feynman--Kac
 formulae~II,
 {\em Publ.\ Res.\ Inst.\ Math.\ Sci.}~{\bf 20} (1984) no.~3,
 607--633.

\bibitem{HP84}
 Hudson, R. L.; Parthasarathy, K. R.:
 Quantum It\^o's formula and stochastic evolutions,
 {\em Comm.\ Math.\ Phys.}~{\bf 93} (1984) no.~3, 301--323.

\bibitem{Lin05} Lindsay, J. M.:
 Quantum stochastic analysis --- an introduction,
 in:
 {\em Quantum Independent Increment Processes I} (2005) 181--271,
 Lecture Notes in Mathematics {\bf 1865}, Springer, Berlin.

\bibitem{LiS97} Lindsay, J. M.; Sinha, K. B.:
 Feynman--Kac representation of some noncommutative elliptic
 operators,
 {\em J.\ Funct.\ Anal.}~{\bf 147} (1997) no.~2, 400--419.

\bibitem{LiW00a} Lindsay, J. M.; Wills, S. J.:
 Existence, positivity and contractivity for quantum stochastic flows
 with infinite dimensional noise,
 {\em Probab.\ Theory Related Fields}~{\bf 116} (2000) no.~4,
 505--543.

\bibitem{LiW00b} Lindsay, J. M.; Wills, S. J.:
 Markovian cocycles on operator algebras adapted to a Fock filtration,
 {\em J.\ Funct.\ Anal.}~{\bf 178} (2000) no. 2, 269--305.

\bibitem{PaS83} Parthasarathy, K. R.; Sinha, K. B.:
 A stochastic Dyson series expansion,
 in: \emph{Theory and Application of Random Fields (Bangalore, 1982)}
 (1983) 227--232,
 Lecture Notes in Control and Information Science
 \textbf{49}, Springer, Berlin.

 \bibitem{RSv2} Reed, M.; Simon, B.:
 {\em Methods of Modern Mathematical Physics II.
 Fourier Analysis, Self-Adjointness},
 Academic Press, New York, 1975.

 \bibitem{Simon} Simon, B.:
 {\em Functional Integration and Quantum Physics},
 Academic Press, New York, 1979.

\end{thebibliography}

\end{document}